\documentclass[12pt,twoside ]{article}
\usepackage{graphicx}% Pour les figures
\usepackage[french,english]{babel}
\usepackage[utf8]{inputenc} % Pour utiliser les caractères accentués
\usepackage[T1]{fontenc} % Pour utiliser les caractères accentués
\usepackage{array} %pour faire des tableaux
\usepackage{tabularx,booktabs,multirow} %pour faire des tableaux, tables, etc.
\usepackage{amsmath, amsfonts, amssymb, amsthm, mathtools, thmtools, mathrsfs} %outils de math
\usepackage{ dsfont } %lettres doubles math
\usepackage{lmodern} %caractères spéciaux
\usepackage{ytableau, youngtab, genyoungtabtikz} %pour tab de young
\ytableausetup{aligntableaux=center}
\usepackage{cancel} %pour barrer du texte
\usepackage{comment}%pour commenter plusieurs lignes de code
\usepackage{url}%pour avoir des liens url dans le texte
\usepackage{subfigure}%pour avoir des sous-figures
\usepackage[normalem]{ulem}%?
\usepackage{makecell}
\usepackage{rotating}

\usepackage{imakeidx}
\makeindex[title=Index des définitions, program=makeindex, columns=2]
\let\emphorig\emph
\renewcommand{\emph}[1]{\emphorig{#1}\index{#1}}

\newtheoremstyle{mes_theoremes}{1.5em}{2em}{}{}{\bfseries}{~:~}{\parskip}{\thmname{#1}\thmnumber{ #2}\thmnote{ (#3)}}
\theoremstyle{mes_theoremes} 
\newtheorem{theo}{Theorem}[section]
\newtheorem{theo*}{Theorem}
\newtheorem{prop}[theo]{Proposition}
\newtheorem{cor}[theo]{Corollary}

\newtheorem{rem}[theo]{Remark}

\newtheorem*{ex*}{Example}
\newtheorem{ex}[theo]{Example}

\usepackage[colorinlistoftodos]{todonotes}
\newcommand{\flo}[1]{\todo[size=\tiny,color=green!30,inline]{#1}}
\newcommand{\etienne}[2][]{\todo[size=\tiny,color=red!30,inline,#1]{#2}}

\newcommand{\Q}{\mathbb{Q}}

\newcommand{\C}{\mathbb{C}}

\newcommand{\isom}{\simeq}

\Yboxdim{0.4cm} %Pour que la taille des cases soit 0.5cm
\newcommand\rd{\Yfillcolour{red}}

\newcommand\wt{\Yfillcolour{white}}

\usepackage{tikz}

\begin{document}

%%%%%%%%%%%%%%%%%%%%
% Pour la page titre
%%%%%%%%%%%%%%%%%%%%
\title{Splitting the square of homogeneous and elementary functions into their symmetric and anti-symmetric parts}

\author{Florence Maas-Gariépy, Étienne Tétreault }

\thispagestyle{empty}        % La page titre n'est pas numérotée
\maketitle

%%%%%%%%%%%%%%%%%%%%
% Document principal
%%%%%%%%%%%%%%%%%%%%

\section{Introduction}

\begin{comment}
\subsection{Début intro étienne}
Let $V$ be a finite-dimensional complex vector space, and let $W$ be a representation of $GL(V)$. Then, we can write $W^{\otimes 2}$ as the direct sum of its symmetric part and its antisymmetric part, denoted $S^2(W)$ and $\Lambda^2(W)$, respectively. While each are irreducible representations of $GL(W)$, they are not as representations of $GL(V)$. The problem is to decompose each into irreducible components.

As a mean of studying this question and many more, Littlewood introduced in 1936 the plethysm of symmetric functions, denoted $f[g]$. In this language, the problem is to write the plethysms $h_2[g]$ and $e_2[g]$ in terms of Schur functions, where $g$ is the character of $W$. When such a decompostion for $g^2$ is known, the problem comes down to determine which Schur functions belong to each plethysm. 

Carré and Leclerc \todo{référence} solved this problem when $g$ is a Schur function. Our contribution is to solve the problem for two other linear basis of symmetric functions, namely the complete homogeneous symmetric functions $h_{\lambda}$ and the elementary symmetric functions $e_{\lambda}$. 

\subsection{Début intro Flo}
\end{comment}

The composition of polynomial representations of $GL_n(\C)$ translate in terms of characters to the plethysm of symmetric functions. Although introduced in 1936 by Littlewood \cite{Littlewood}, the plethysm of two symmetric functions $f$ and $g$, denoted $f[g]$, still carries a lot of open problems.

One open problem is understanding the decomposition of $f[g]$ in the basis of Schur functions $s_\nu$, since they correspond to the irreducible characters for polynomial representations of $GL_n(\mathbb{C})$. Basic properties of plethysm %, and the decomposition of all symmetric functions in the basis of Schur functions 
reduce the  %"simpler" 
problem (slightly) to understanding the decomposition of $f$, and of 
$s_\mu[g]$, in the Schur basis. %, and the plethysms $s_\mu[g]$
%where $s_{\mu}$ is a Schur function. %in the basis of Schur functions.

This has applications in representation theory (and so chemistry, physics, crystallography, etc.), and in geometric complexity theory (see \cite{Mulmuley} for an introduction).

For a specific symmetric function $g$, one wants %An important function 
to study %in this context is 
$g^n$, which decomposes as a sum of plethysms $s_\mu[g]$ for $\mu\vdash n$, each with multiplicity $f^\mu$ counting standard tableaux of shape $\mu$.
%
%This open problem has very partial results yet. Most of the time, 
%For most symmetric functions $g$, 
Often, we also know the decomposition of $g^n$ %in term of plethysms $s_\mu[g]$ and its decomposition 
in the basis of Schur functions, so the real difficulty is to discriminate how many copies of $s_\lambda$ contribute to a given plethysm $s_\mu[g]$. For small %partitions $\mu$ 
values of $n$, this is computable, but rather slow as the calculation of the coefficients appearing in these decompositions has been proven to be $\# P$-hard \cite{FischerIkenmeyer}.

In the case $n=2$ that interests us, the plethysms $s_2[g]$ and $s_{11}[g]$ in $g^2$ are often referred to as the symmetric and anti-symmetric parts of the square. %$g^2$.%, since this corresponds to the actual decomposition of tensors of modules into two submodules, one with every element symmetric, and the other, anti-symmetric. This is the case that interest us.

For $g=s_\lambda$ a Schur function, % and $n = 2$, 
Carré and Leclerc \cite{CarreLeclerc} gave a combinatorial solution to this problem. % when $g$ is a Schur function. 
The decomposition of the plethysms of $s_\lambda^n$ in the basis of Schur function remains however open for $n>2$.

Our contribution is to give a combinatorial description %for $n = 2$ 
in the case of two other linear bases of symmetric functions, namely the complete homogeneous symmetric functions $h_{\lambda}$ and the elementary symmetric functions $e_{\lambda}$. As the construction for both is similar, we first describe our construction for $h_\lambda$ and %explain in section~\ref{section:e} how this construction can be 
extended it to $e_\lambda$ afterwards. We use a combinatorial description of $h_\lambda$ in terms of %the tabloids of shape $\lambda$ appearing in $h_\lambda^2$ (resp. $e_\lambda^2$), which we decompose into 
$\lambda$-tuples of row tableaux (corresponding to tabloids), and use the RSK algorithm to describe the $Q$-tableaux %counted by $K_{\lambda^2}^\nu$ and 
indexing the copies of $s_\nu$ appearing in the decomposition of $h_\lambda^2$ in the basis of Schur functions. A sign statistic on these $Q$-tableaux determines to which plethysm the associated Schur function contributes. Our description of the RSK algorithm uses the product of tableaux $\ast$ of Lascoux and Schützenberger \cite{LascouxSchutzenberger}, formalized by Fulton in \cite{Fulton}, and rectification (Rect) of skew tableaux using jeu de taquin. Our proofs rely on this description, as well as basic properties of plethysm.% and of symmetric functions.

Our main results are the following. For conciseness, we combined here the analogous results for $h_\lambda$ and $e_\lambda$.

\begin{theo*}
Fix a partition $\lambda = (\lambda_1, \ldots, \lambda_\ell)$.

\begin{itemize}
\item Let $Q$ be a tableau of shape $\nu\vdash 2|\lambda|$ and content $\lambda^2$. % indexing a Schur function $s_\nu$ in $h_\lambda^2$. 
4Let $Q^{(i)}$ be the subtableau of $Q$ containing entries $2i-1$ and $2i$, for $1\leq i \leq \ell$ and $Q_i = \text{Rect}(Q^{(i)})$ of shape $(2\lambda_i-j_i, j_i)$. Define the sign of $Q$ to be $\text{sign}(Q) = \prod_{i=1}^{\ell} (-1)^{j_i}$.\\

The copy of $s_\nu$ indexed by $Q$ in $h_\lambda^2$ lies in $s_2[h_\lambda]$ if $\text{sign}(Q) = 1$, and in $s_{11}[h_\lambda]$ if $\text{sign}(Q) = -1$.

\item Let $\widetilde{Q}$ be a conjugate tableau of shape $\nu\vdash 2|\lambda|$ and content $\lambda^2$. % indexing a Schur function $s_\nu$ in $h_\lambda^2$. 
Let $\widetilde{Q^{(i)}}$ be the subtableau of $\widetilde{Q}$ containing entries $2i-1$ and $2i$, for $1\leq i \leq \ell$ and $\widetilde{Q_i} = (\text{Rect}(\widetilde{Q^{(i)}}'))'$ of shape $(2^{\lambda_i-j_i}, 1^{2j_i})$. Define the (conjugate) sign of $\widetilde{Q}$ to be $\text{sign}_c(\widetilde{Q}) = \prod_{i=1}^{\ell} (-1)^{j_i}$.\\

The copy of $s_\nu$ indexed by $\widetilde{Q}$ in $e_\lambda^2$ lies in $s_2[e_\lambda]$ if $\text{sign}_c(\widetilde{Q}) = 1$, and in $s_{11}[e_\lambda]$ if $\text{sign}_c(\widetilde{Q}) = -1$.
\end{itemize}
\label{thm:signTableauQ}
\end{theo*}

This gives us the following decomposition :

\begin{theo*}
Consider the following numbers:
\begin{itemize}
    \item $(K^{\nu}_{\lambda^2})^{+} = \#\{ Q \text{ tableau of shape } \nu \text{ and content } \lambda^2 \ | \ \text{sign}(Q)=1 \}$ 
    \item $(K^{\nu}_{\lambda^2})^{-} = \#\{ Q \text{ tableau of shape } \nu \text{ and content } \lambda^2 \ | \ \text{sign}(Q)=-1 \}$ 
    \item $(K^{\nu'}_{\lambda^2})^{+} = \#\{ \widetilde{Q}' \text{ tableau of shape } \nu' \text{ and content } \lambda^2 \ | \ \text{sign}_c(\widetilde{Q})=1 \}$ 
    \item $(K^{\nu'}_{\lambda^2})^{-} = \#\{ \widetilde{Q}' \text{ tableau of shape } \nu' \text{ and content } \lambda^2 \ | \ \text{sign}_c(\widetilde{Q})=-1 \}$ 
\end{itemize}

Then we have: 

\vspace{-2em}
\begin{align*}
s_2[h_{\lambda}] &= \displaystyle \sum_{\nu} (K^{\nu}_{\lambda^2})^{+} s_{\nu} \\
s_{11}[h_{\lambda}] &= \displaystyle \sum_{\nu} (K^{\nu}_{\lambda^2})^{-} s_{\nu} \\
s_2[e_{\lambda}] &= \displaystyle \sum_{\nu} (K^{\nu'}_{\lambda^2})^{+} s_{\nu} \\
s_{11}[e_{\lambda}] &= \displaystyle \sum_{\nu} (K^{\nu'}_{\lambda^2})^{-} s_{\nu} 
\end{align*}
\label{thm:SignTableau}
\end{theo*}

In section 2, we recall definitions.

In section 3, we introduce all concepts necessary to state and prove our theorems for $h_\lambda$. We recall the product of tableaux of Lascoux, Schützenberger \cite{LascouxSchutzenberger} and Fulton \cite{Fulton}, which uses jeu de taquin, and define the $\text{RSK}$ algorithm using this product of tableaux. %We use some crystal theory (see \cite{BumpSchilling} for an introduction) to show this product is well defined, using a result of van Leeuwen \cite{VanLeeuwen2} that jeu de taquin commutes with crystal operators. We finally use properties of plethysm and symmetric functions to show the result in terms of symmetric functions, then translate it to $\lambda$-tuples of line (or column) tableaux.
We show that it gives a bijection between $\lambda$-tuples of row tableaux and pairs of tableaux of the same shape $\nu$, with the recording tableau $Q$ having content $\lambda$. We finally define a sign statistic on tableaux of content $\lambda^2$. We use basic properties of plethysm and symmetric functions, and properties of the RSK algorithm, to prove theorems 1 and 2 for $h_\lambda$.

\setcounter{theo*}{0}

In section 4, we translate our construction to $e_\lambda$. 

Finally, in section 5, we show how our results could be generalized to higher plethysms. We also describe another construction using that of Carré-Leclerc, which gives an alternative proof of our results in terms of tuples of domino tableaux. This construction offers hope of generalization to higher plethysms in terms of ribbon tableaux.

\section{Definitions}

%\subsection{tableaux}

%A \textit{composition} of an integer $n$ is sequence $\lambda = (\lambda_1, \lambda_2, \dots)$ of positive integers that sums up to $n$. When $\lambda_1 \geq \lambda_2 \geq \dots$, we say that $\lambda$ is a \textit{partition} of $n$. We write $\lambda \vdash n$ when $\lambda$ is a partition of $n$. The number of parts of a partition $\lambda$, called its $\textit{length}$, is denoted $\ell(\lambda)$. We associate to a partition its \textit{Ferrers diagram}, which is 
%\todo{Notation anglaise ou française?} 

\subsection{Tableaux}

We recall here definitions and notations on tableaux. For a more detailed introduction, see \cite{Fulton} or \cite{Sagan}.

Recall that a partition is a weakly decreasing vector of positive integers $\lambda = (\lambda_1,\lambda_2, \ldots, \lambda_\ell)$. We identify a partition $\lambda$ to its diagram, the collection of boxes both top- and left-justified with $\lambda_i$ boxes in its $i^{th}$ row, counting top to bottom. %We denote $\lambda$ both the partition and the Ferrers diagram.
We denote $\lambda^2 = (\lambda_1, \lambda_1,\lambda_2,\lambda_2, \ldots, \lambda_\ell, \lambda_\ell)$ and $\lambda'$ the conjugate of $\lambda$, where rows of $\lambda$ become the columns of $\lambda'$. We denote also $|\lambda| = \sum_i \lambda_i$, $\lambda\vdash n$ if $\sum_i \lambda_i = n$ and $\ell(\lambda) = \ell$ the number of parts of $\lambda$.

A \textit{semistandard tableau} $t$ of shape $\lambda$ is the filling of $\lambda$ with positive integers, such that rows weakly increase from left to right, and columns strictly increase from top to bottom. We denote $\text{SSYT}(\lambda)$ the set of tableaux of shape $\lambda$. %From now on, w
We use \textit{tableau} to mean \textit{semistandard tableau}, unless stated otherwise.

The \textit{conjugate} of a tableau $t$ of shape $\lambda$ has shape $\lambda'$ and is obtained by reflecting $t$ along its main diagonal. A conjugate tableau then has strictly increasing rows and weakly increasing columns.

We allow tableaux to have skew shapes $\lambda/\mu$, where the boxes of the partition $\mu$ are left blank. The content of a tableau $t$ is the vector $(\alpha_1,\alpha_2,\ldots,\alpha_k)$, where $\alpha_i$ counts the number of entries $i$ in $t$.

\begin{ex}
The tableau $\gyoung(:::;1;1;2;2;3;4,::;1;2;2;3;3,;2;3;3;3)$ has (skew) shape $(9,7,4)/(3,2)$ and content $(3,5,6,1)$.
\label{ex:tab}
\end{ex}

%Note that 
$\text{SSYT}(n)$ correspond to row tableaux of length $n$, and $\text{SSYT}(1^n)$ to column tableaux of height $n$. 
We define $\lambda$-tuples of row tableaux to be elements of the cartesian product  $T_{\leq}(\lambda) := \text{SSYT}(\lambda_1) \times \text{SSYT}(\lambda_2) \times \ldots \times \text{SSYT}(\lambda_{\ell})$, and we define $\lambda$-tuples of column tableaux to be elements of $T_{\wedge}(\lambda) := \text{SSYT}(1^{\lambda_1}) \times \text{SSYT}(1^{\lambda_2}) \times \ldots \times \text{SSYT}(1^{\lambda_{\ell}})$. The content of a tuple of tableaux is the sum of the contents of each tableau forming the tuple.

\begin{ex}
The tuples $s$, $t$ below are $(4,3)^2$-tuples of tableaux (row/column).% tableaux.%, for $\lambda = (4,3)$.
\[s = (s_1, s_2, s_3, s_4) = (\young(1345), \young(2224), \young(122), \young(137)) \in T_\leq ((4,3)^2),\]
\[t = (t_1, t_2, t_3, t_4) = (\young(1,3,4,5), \young(2,4,7,8), \young(1,2,3), \young(1,3,7)) \in T_{\wedge} ((4,3)^2).\] 
%$s$ corresponds to the tabloid $ \young(1345,2224,112,137)$, % with $s_i$ being the $i^{th}$ line, 
%and $t$, to the column tabloid \young(1211,3423,4737,58). 
Their contents are $(4,5,2,2,1,0,1)$ and $(3,2,3,2,1,0,2,1)$.
\label{ex:tupletab}
\end{ex}

The content allows one to associate a monomial $x^t$ with a tableau (or $\lambda$-tuple), where the power of $x_i$ is given by $\alpha_i$, the number of entries $i$ in $t$. %The monomial associated to a $\lambda$-tuple of line (or column) tableaux is the product of the monomials of each of its line (or column) tableaux.

To a tableau we can also associate a \textit{reading word}, by reading off its entries from left to right and bottom to top. The reading word of a tuple of tableaux is the concatenation of the reading words of each tableau. For example, the tableau $t$ in example~\ref{ex:tupletab} has reading word $54318742321731$, and associated monomial $x^t = x_1^3x_2^2x_3^3x_4^2x_5x_7^2x_8$.

%A \textit{tabloid} of shape $\lambda$ is the content of $\lambda$ with positive integers such that lines weakly increase from left to right. If all its lines strictly increase, we say that a tabloid is line strict. Tableaux can be seen as tabloids, but the converse is generally not true. However, tabloids can be seen as a collection of line tableaux of respective lengths $\lambda_i$, and line strict tabloids as a collection of column tableaux of respective heights $\lambda_i$. The content of a tabloid $t$ is defined in the same way as that of a tableau.\\

%The tabloid $t = \youngtabloid(112223,11223,2333)$ has shape $(6,5,4)$ and is not line strict.\\

%Lets denote $\text{SSYT}(\lambda)$ the set of tableaux of shape $\lambda$, $T_\leq (\lambda)$ the set of tabloids of shape $\lambda$, and $T_< (\lambda)$ its subset of line strict tabloids.

\subsection{Symmetric functions}

For an introduction to symmetric functions, see \cite{Macdonald} or \cite{Sagan}.

Recall that \textit{symmetric functions} over $\Q$ are (potentially infinite) polynomials $f(x)$ on formal variables $x = (x_1, x_2, x_3, \ldots)$, such that permuting any two variables gives back the same polynomial. They form a well studied graded ring $\Lambda$ for which many basis are known, including the following.

For each partition $\lambda$, the \textit{Schur function} $s_\lambda$ can be defined in terms of tableaux :
\[
s_\lambda = \sum_{t \in \text{SSYT}(\lambda)} x^t.
\]

The \textit{homogeneous symmetric functions} $h_\lambda$ can be defined in terms of $\lambda$-tuples of row tableaux :

\[
h_\lambda = \sum_{t \in T_{\leq}(\lambda)} x^t = h_{\lambda_1} h_{\lambda_2} \ldots h_{\lambda_\ell} = s_{(\lambda_1)} s_{(\lambda_2)} \ldots s_{(\lambda_\ell)}.
\]
%Note that $h_\lambda = h_{\lambda_1} h_{\lambda_2} \ldots h_{\lambda_\ell} = s_{\lambda_1} s_{\lambda_2} \ldots s_{\lambda_\ell}$.

Similarly, the \textit{elementary symmetric functions} $e_\lambda$ can be defined in terms of $\lambda$-tuples of column tableaux:
\[
e_\lambda = \sum_{t\in T_{\wedge}(\lambda)} x^t = e_{\lambda_1} e_{\lambda_2} \ldots e_{\lambda_\ell} = s_{(1^{\lambda_1})} s_{(1^{\lambda_2})} \ldots s_{(1^{\lambda_\ell})}.
\]

The \textit{power sum symmetric functions} $p_\lambda$ are defined as the product $p_{\lambda_1} p_{\lambda_2} \ldots p_{\lambda_\ell}$, where
\[ p_n = \sum_{i\geq 1} x_i^n.
\]

The transition matrices between these bases are known, and important problems in representation theory correspond to decomposing a symmetric function in term of one of these bases. 
A widely studied problem is to find combinatorial descriptions of the coefficients appearing in the decomposition of the product of two symmetric functions, which is also a symmetric function. Generally we want the decomposition in the basis of the Schur functions,
since they correspond to irreducible representations. The following propositions are well-known results about this problem, and can be found for example in \cite{Macdonald}. %, and so the decomposition into the basis of Schur functions gives a decomposition of a representation into irreducible representations.

\begin{comment}
\flo{Je garderais peut-être juste règle de pieri...}
For Schur functions, Littlewood and Richardson \cite{LittlewoodRichardson} stated the following result in 1934 (although a valid proof was only provided in 1974 by Thomas \cite{Thomas}):

\begin{prop}[Littlewood-Richardson rule]
$s_\mu s_\nu = \sum_\lambda c_{\mu \nu}^{\lambda} s_\lambda $, where the Littlewood-Richardson coefficients $c_{\mu \nu}^\lambda$ count the number of tableaux of shape $\lambda/\mu$, of content $\nu$ (or shape $\lambda/\nu$ and content $\mu$), which are Yamanouchi. 
\end{prop}
\end{comment}

%This rule is a generalization of the Pieri rule:
The Pieri rule describes the product of any Schur function by a homogeneous or elementary symmetric function :

\begin{prop}[Pieri rule]
For any partition $\lambda$ and integer $n$, we have:
\begin{itemize}
\item $s_{\lambda}s_{(n)} = s_{\lambda}h_n = \displaystyle \sum_{\nu} s_{\nu}$,
where the sum is over all partitions $\nu$ obtained by adding $n$ boxes to $\lambda$, no two in the same column.
\item $s_{\lambda}s_{(1^n)} = s_{\lambda}e_n = \displaystyle \sum_{\nu} s_{\nu}$,
where the sum is over all partitions $\nu$ obtained by adding $n$ boxes to $\lambda$, no two in the same row.
\end{itemize}
\end{prop}

This allows us to describe the decomposition of $h_n^2$ and $e_n^2$ : % into the basis of Schur functions :

\begin{prop}
\[(h_n)^2 = \sum_{j=0}^n s_{(2n-j,j)} \qquad \text{and} \qquad (e_n)^2 = \sum_{j=0}^n s_{(2^{n-j},1^{2j}}).\]
\end{prop}

Iteratively applying the Pieri rule to get a description of $h_\lambda^2$ and $e_\lambda^2$, we obtain:

\begin{prop}\label{Power}
\[
(h_{\lambda})^2 = h_{\lambda^2} = \displaystyle \sum_{\nu \vdash 2|\lambda|} K_{\lambda^2}^{\nu} s_{\nu} \qquad \text{and} \qquad (e_{\lambda})^2 = e_{\lambda^2} = \displaystyle \sum_{\nu \vdash 2|\lambda|} K_{\lambda^2}^{\nu'} s_{\nu},
\]
where the Kostka numbers $K_{\lambda^2}^\nu$ count the number of tableaux of shape $\nu$ and content $\lambda^2$.%\\
%We also have:
%$(e_{\lambda})^2 = e_{\lambda^2} = \displaystyle \sum_{\nu \vdash n|\lambda|} \widetilde{K_{\lambda^2}^{\nu}} s_{\nu}$, where $\widetilde{K_{\lambda^2}^\nu}$ counts the number of conjugate tableaux of shape $\nu$ and content $\lambda^2$.
\end{prop}

\section{Associating a $Q$-tableau indexing a copy of $s_\nu$ in $h_\lambda^2$ to a plethysm}

%As described above, we have that $ h_\lambda^2 = \displaystyle \sum_\nu K^{\nu}_{\lambda^2} s_\nu$.%, where $K^{\nu}_{\lambda^2}$ is the number of tableau of shape $\nu$ and content $\lambda^2$.%, and that $e_\lambda^2 = \displaystyle \sum_\nu K^{\nu'}_{\lambda^2} s_{\nu}$, where $K^{\nu'}_{\lambda^2}$ is the number of conjugate tableau of shape $\nu$ and content $\lambda^2$. %$ = \sum_\nu \sum_{Q\in \text{SSYT}(\nu,\lambda^2)} s_\nu$.

%Moreover, theorem \ref{thm:ProductTabloids} describes the monomials of each plethysm in $h_\lambda^2$ or $e_\lambda^2$. %considered in this article. 
%It remains to express these sums of monomials in terms of Schur functions.

We will see that the tableaux of shape $\nu$ and content $\lambda^2$ counted by $K_{\lambda^2}^{\nu}$ in $ h_\lambda^2 = \sum_\nu K^{\nu}_{\lambda^2} s_\nu$ are $Q$-tableaux obtained through the $\text{RSK}$ algorithm. %The definition of RSK is given in terms of product of tableaux, which is central to the proof. 
%In order to apply RSK, we will need to describe a bijection between $\lambda^2$-tuples of line tableaux appearing in each sum and biwords.

\subsection{Product of tableaux, jeu de taquin and RSK}

We define the RSK algorithm in terms of product of tableaux, which will be central to our construction. The product of tableaux was introduced by Lascoux and Schützenberger in the setting of the plactic monoid \cite{LascouxSchutzenberger}. Its elements correspond to tableaux: a tableau is identified with the Knuth-equivalence class of the reading word of the tableau. The product of tableaux was more formally defined by Fulton in \cite{Fulton} using \textit{jeu de taquin}, so that it corresponds to the product on words (concatenation) in the plactic monoid.

Jeu de taquin gives a way to rectify a skew tableau of shape $\lambda/\mu$ to a straight tableau. A jeu de taquin slide starts at an inner corner of $\lambda/\mu$ and exchanges the empty cell %from an inner corner of $\mu$ 
with a non-empty adjacent cell, respecting constraints on the entries of the rows and columns, until the empty cell lies on the outer border. The rectified tableau is independent of the order of the slides. \\

Let's define the product of two tableaux $t_1, t_2$, of respective shape $\mu, \nu$ :
\begin{enumerate}
\item Construct the skew tableau $t_1 \ast t_2$ by placing $t_1$ below and left of $t_2$, into the skew shape $(\mu_1+\nu_1, \mu_1+\nu_2, \ldots, \mu_1+\nu_\ell, \mu_1, \mu_2, \ldots, \mu_k)/(\mu_1^\ell)$. %, such that $t_1$ lies in the bottom left corner, and $t_2$ in the top right corner. %This skew tableau is such that no lines or columns of $t_1,t_2$ intersect.
%The reading word of $t_1\ast t_2$ is the concatenation of the reading word of both. 
\item The product $T$ of $t_1,t_2$ is the rectification of $t_1\ast t_2$ using jeu de taquin.
\end{enumerate}

\begin{ex}
For $t_1 = \young(1122333,2334,445)$ and $t_2 = \young(12222,233,34,5)$,\\
$t_1\ast t_2 = \gyoung(:::::::;1;2;2;2;2,:::::::;2;3;3,:::::::;3;4,:::::::;5,;1;1;2;2;3;3;3,;2;3;3;4,;4;4;5)$ \hspace{-3em}, and $T = \text{Rect}(t_1\ast t_2) = \young(111222222,22333333,3444,45,5)$.\\

Three jeu de taquin slides are illustrated below.
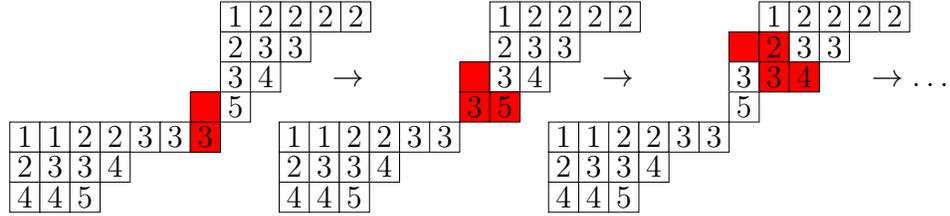
\begin{center}
   \gyoung(:::::::;1;2;2;2;2,:::::::;2;3;3,:::::::;3;4,::::::!\rd;!\wt;5,;1;1;2;2;3;3!\rd;3,!\wt;2;3;3;4,;4;4;5) \hspace{-2em} \raisebox{1em}{$\rightarrow $} \hspace{-3.5em} \gyoung(:::::::;1;2;2;2;2,:::::::;2;3;3,::::::!\rd;!\wt;3;4,::::::!\rd;3;5,!\wt;1;1;2;2;3;3,;2;3;3;4,;4;4;5) \hspace{-2em} \raisebox{1em}{$\rightarrow $} \hspace{-3.5em} \gyoung(:::::::;1;2;2;2;2,::::::!\rd;;2!\wt;3;3,::::::;3!\rd;3;4,!\wt::::::;5,;1;1;2;2;3;3,;2;3;3;4,;4;4;5) \hspace{-2em} \raisebox{1em}{$\rightarrow  \ldots $} %\hspace{-4em} \gyoung(::::::!\rd;;1;2;2;2;2!\wt,::::::;2;3;3;3,::::::;3;4,::::::;5,;1;1;2;2;3;3,;2;3;3;4,;4;4;5)
\end{center}
\end{ex}

\begin{prop}
Let $t_1, t_2$ be tableaux of respective shape $\mu$, $\nu$. The word obtained by the concatenation of $\text{word}(t_1)$ and $\text{word}(t_2)$ lies in the plactic class of the reading word of $T = \text{Rect}(t_1\ast t_2)$. %They respectively contribute monomials $x^{t_1}$ and $x^{t_2}$ to $s_\mu$ and $s_\nu$. Their product $T = Rect(t_1\ast t_2)$ contributes $x^T = x^{t_1}x^{t_2}$ to $s_\lambda$ if $shape(T)=\lambda$.
\end{prop}

\begin{proof}[Idea of proof]
Applying jeu de taquin on a skew tableau takes its reading word onto a Knuth-equivalent reading word. Since the reading word of $t_1\ast t_2$ is equal to $\text{word}(t_1)\text{word}(t_2)$, we have the desired result.
\begin{comment}
In the tensor product crystal $B(\mu)\otimes B(\nu)$, the elements are (formal) tensors of tableaux, with crystal operators acting on the concatenation of both reading words. Van Leeuwen proved that jeu de taquin commutes with crystal operators in \cite{VanLeeuwen2}. This means that the effect of a crystal operator %$f_i$ 
will be the same on $t_1\otimes t_2$, $t_1\ast t_2$ and $T = Rect(t_1\ast t_2)$. $B(\mu)\otimes B(\nu)$ decomposes into connected components, with $c_{\mu, \nu} ^{\lambda}$ connected components isomorphic to $B(\lambda)$, as per the Littlewood-Richardson rule. Since $T$ has a straight shape, it lies in a connected component $B(\lambda)$ regrouping all distinct tableaux of straight shape $\lambda$. Since the tableau product commutes with crystal operators, the connected component containing $t_1\otimes t_2$ and that of $T = Rect(t_1\ast t_2)$ ($B(\lambda)$) will be isomorphic. So $t_1\otimes t_2$ lies in a connected component of $B(\mu)\otimes B(\nu)$ isomorphic to $B(\lambda)$. Since each connected component $B(\lambda)$ corresponds to a copy of $s_\lambda$, then $x^T = x^{t_1}x^{t_2}$ contributes to $s_\lambda$.
\end{comment}
\end{proof}

The RSK algorithm \cite{Knuth} is closely linked to the plactic monoid, since all words in the same (plactic) class will have the same insertion tableau under RSK. %We give here an alternate definition of the algorithm in terms of product of tableaux, which will be central to our argument. 
Recall that RSK establishes a bijection between biwords $W = \big( $~\begin{tabular}{c} $u_1 \ u_2 \ \ldots \ u_k$ \\ $v_1 \ v_2 \ \ldots \ v_k$ \\ \end{tabular} \big) = \big( \begin{tabular}{c} $u$ \\ $v$ \\ \end{tabular} $\big)$ and pairs of tableaux $(P,Q)$ of the same shape $\mu$, for $\mu\vdash k$, with $\text{content}(P) = \text{content}(v)$, and $\text{content}(Q) = \text{content}(u)$. Recall that the bi-letters $\big($ \begin{tabular}{c} $u_i$ \\ $v_i$ \\ \end{tabular} $\big)$ of a biword are ordered such that the $u_i$'s weakly increase, and for $u_i = u_{i+1}$, $v_i\leq v_{i+1}$.

%Biwords are two-line arrays of positive integers $W = \big( \begin{tabular}{c} u_1 u_2 \ldots u_k \\ v_1 v_2 \ldots v_k \\ \end{tabular} \big)$ 
%$ = \big( \begin{tabular}{c} u \\ v \\ \end{tabular} \big) $
%, with t
Note that words $w$ are inserted by RSK into pairs $(P,Q)$ by considering $u = 123\ldots \text{length}(w)$ and $v = w$. Then the tableau $Q$ is standard. \\

Our alternate definition of RSK using product of tableaux %for the insertion 
goes as follows.\\
Let $P_i$ the tableau obtained after the insertion of $v_i$, with $P_0 = \emptyset$.
\Ystdtext1
\Yboxdim{25pt}
Then $P_{i+1} = \text{Rect}(P_i\ast$ \young(<$v_{i+1}$>)$)$, and $P = P_k$.
\Yboxdim{0.4cm}
The tableau $Q$ records the order in which cells are added,  %in the RSK insertion, 
with entry $u_i$ in position $\text{shape}(P_{i+1})/\text{shape}(P_i)$.

\begin{rem}
%The definition of RSK in terms of product of tableaux can be seen as putting all letters of $v$ on an anti-diagonal into $v_1 \ast v_2 \ast \ldots \ast v_k$. %, such that the reading word of that skew tableau is $v$. 
%Starting off with $v_1\ast v_2 \ast \ldots \ast v_k$, and 
This process can be seen as rectifying $v_1\ast v_2 \ast \ldots \ast v_k$ one cell after the other from left to right to yield the tableaux $P_i$, as illustrated below. %are obtained by rectifying the $i$ first cells. Rectifying one cell after the other from left to right and 
Recording the added cells with the $u_i$'s yields the tableaux $Q_i$. 
%Since all words in the same plactic class map onto the same $P$-tableau under RSK, and jeu de taquin preserves Knuth-equivalence, then this description is valid. 

\vspace{-1em}
\gyoung(:::::::::::::;3,::::::::::::;2,:::::::::::;1,::::::::::;2,:::::::::;1,::::::::;1,:::::::;3,::::::;3,:::::;2,::::;1,:::;4,::;3,:;2,;1) \hspace{-6em} $\rightarrow$ \hspace{-6em}
\gyoung(:::::::::::::;3,::::::::::::;2,:::::::::::;1,::::::::::;2,:::::::::;1,::::::::;1,:::::::;3,::::::;3,:::::;2,::::;1,:::;4,::;3,;1;2,:) \hspace{-5em} $\rightarrow$ \hspace{-6em}
\gyoung(:::::::::::::;3,::::::::::::;2,:::::::::::;1,::::::::::;2,:::::::::;1,::::::::;1,:::::::;3,::::::;3,:::::;2,::::;1,:::;4,;1;2;3,:,:) \hspace{-5em} $\rightarrow$ \hspace{-6em}
\gyoung(:::::::::::::;3,::::::::::::;2,:::::::::::;1,::::::::::;2,:::::::::;1,::::::::;1,:::::::;3,::::::;3,:::::;2,::::;1,;1;2;3;4,:,:,:) \hspace{-5em} $\rightarrow$ \hspace{-5em} 
\gyoung(::::::::::::;3,:::::::::::;2,::::::::::;1,:::::::::;2,::::::::;1,:::::::;1,::::::;3,:::::;3,::::;2,;1;1;3;4,;2,:,:,:) \hspace{-5em} $\rightarrow$ $\ldots$ %\hspace{-5em} 
%\gyoung(:::::::::::;3,::::::::::;2,:::::::::;1,::::::::;2,:::::::;1,::::::;1,:::::;3,::::;3,;1;1;2;4,;2;3) \hspace{-3em} $\rightarrow$ \hspace{-4em} 
%\gyoung(::::::::::;3,:::::::::;2,::::::::;1,:::::::;2,::::::;1,:::::;1,::::;3,;1;1;2;3,;2;3;4) \hspace{-3em} $\rightarrow$ \hspace{-4em} 
%\gyoung(::::::::::;3,:::::::::;2,::::::::;1,:::::::;2,::::::;1,:::::;1,;1;1;2;3;3,;2;3;4) \hspace{-3em} $\rightarrow$ \hspace{-4em}
%\gyoung(:::::::::;3,::::::::;2,:::::::;1,::::::;2,:::::;1,;1;1;1;3;3,;2;2;4,;3) \hspace{-3em} $\rightarrow$ \hspace{-4em}
\end{rem}

\subsection{Associating a $\lambda^2$-tuple of row tableaux to a $Q$-tableau}

%The RSK algorith identifies $\lambda^2$-tuples of line tableaux with pairs of tableaux, by first defining a bijection between $\lambda$-tuples of line tableaux with biwords :

\begin{prop}\label{prop:RSK}
$\lambda$-tuples of row tableaux with content $\alpha$ are in bijection with %biwords  inserted through the $RSK$ insertion into 
pairs of tableaux of the same shape $(P_t,Q_t)$, where the content of $P_t$ is $\alpha$ and the content of $Q_t$ is $\lambda$.
\end{prop}

\begin{proof}
Consider the biword $W_t$ built from a $\lambda$-tuple of row tableaux $t = (t_1,t_2, \ldots, t_k)$ using the bi-letters $\left(\begin{tabular}{c}i\\j\end{tabular}\right)$, for $j$ an entry in $t_i$. The top word of $W_t$ is $1^{\lambda_1}2^{\lambda_2}\ldots \ell^{\lambda_k}$, and the bottom word is the reading word of $t$.

The map $t \mapsto W_t$ is a bijection between $\lambda$-tuples of row tableaux of content $\alpha$ and biwords such that the top word has content $\lambda$, and the bottom word has content $\alpha$ and is the reading word of the $\lambda$-tuples. By RSK, these biwords are in bijection with pairs $(P_t,Q_t)$ of tableaux with the desired properties. %of the same shape such that $P_t$ has content $\alpha$ and $Q_t$ has content $\lambda$. 
\end{proof}

\begin{ex}
If $t = (\young(1234),\young(1233),\young(112),\young(123))$, then $W_t = \big( \begin{tabular}{cccc} 1111&2222& 333 & 444 \\ 1234 & 1233 & 112 & 123 \\ \end{tabular}\big)$ and $(P,Q) = (\gyoung(;1;1;1;1;1;2;3,;2;2;2;3,;3;3,;4),\young(1111224,22234,33,4) )$.
\end{ex}

\begin{rem}
In the case of a biword associated to a $\lambda$-tuple of row tableaux, the rectifying process of the RSK algorithm can be greatly sped up as illustrated below, by first rectifying together the letters in the same line tableau $t_i$.\\ 

Since all letters of the reading word of $t_i$ correspond to entries $i$ in $Q$, the skew tableau $t_1\ast t_2 \ast \ldots \ast t_\ell$ can be rectified recursively row by row, recording the added cells in the subtableaux of $Q$ with entries $1$ up to $i$. %less than $i+1$.  

\vspace{0.5em}
\hspace{-4em} \gyoung(:::::::::::;1;2;3,::::::::;1;1;2,::::;1;2;3;3,;1;2;3;4) \hspace{-3em} $\rightarrow$  
\gyoung(::::::::;1;2;3,:::::;1;1;2,;1;1;2;3;3,;2;3;4) \hspace{-1.5em} $\rightarrow$ \gyoung(:::::;1;2;3,;1;1;1;1;2,;2;2;3;3,;3;4) \hspace{-1em} $\rightarrow$ \gyoung(;1;1;1;1;1;2;3,;2;2;2;3,;3;3,;4)\\
\hspace{-5em} $Q_4 = \young(1111)$ \hspace{1em} $Q_8 = \young(111122,222)$ \hspace{0.2em} $Q_{11} = \young(111122,2223,33)$ \hspace{0.2em} $Q = \young(1111224,22234,33,4)$
\end{rem}

%\begin{rem}
%We have that $P_t = Rect(t_1\ast t_2 \ast \ldots \ast t_\ell) = Rect(v_1\ast v_2\ast\ldots\ast v_k)$.
%\end{rem}

\begin{rem}
In the case of a $\lambda^2$-tuple of row tableaux, we have a third way to rectify $t_1\ast t_2 \ast \ldots \ast t_{2\ell -1} \ast t_{2\ell}$ into $P_t$ : each pair of tableaux of respective length $\lambda_i$ can be first rectified together into tableaux $T_i$. %, before rectifying them into $P_t$. The tableaux $T_i$ 
They correspond to insertion tableaux for the sub-biwords $\big( $ \begin{tabular}{cccccccc}  2i-1 \ldots 2i-1  & 2i \ldots 2i\\ word($t_{2i-1}$) & word($t_{2i}$) \end{tabular} $\big)$. %We can then consider the pair of tableaux $(T_i,Q_i) = RSK\big( \begin{tabular}{cccccccc} 2i-1 \ldots 2i-1 & 2i \ldots 2i\\ word(t_{2i-1}) & word(t_{2i}) \end{tabular} \big)$.
\end{rem}

We can recover the associated recording tableaux $Q_i$ directly from $Q$ :

\begin{cor}
Let $t$ be a $\lambda^2$-tuple of row tableaux, $W_t$ the associated biword and $(P,Q) = \text{RSK}(W_t)$. Let \[
(T_i,Q_i) = \text{RSK}\big(  \begin{tabular}{cccccccc} 2i-1 \ldots 2i-1 & 2i \ldots 2i \\ \text{word}(
$t_{2i-1}$) & \text{word}($t_{2i}$) \end{tabular} \big),
\] and $Q^{(i)}$ the subtableau of $Q$ with entries $2i-1$ and $2i$, for $1\leq i \leq \ell = \ell(\lambda)$.

Then $Q_i = \text{Rect}(Q^{(i)})$.
\label{theorem:rectificationOfQi}
\end{cor}

This follows from a result found in chapter 5 of \cite{Fulton}.

%\begin{prop}[Fulton, ch.5]\label{prop:Fulton}
%Let $W = \left(\begin{tabular}{c}u\\v\end{tabular}\right)$ a biword and $(P,Q)$ the pair of tableaux of the same shape associated under RSK.\\
%Let $t$ be any tableau of shape $\lambda$. The insertion $t\overset{\text{RSK}}{\leftarrow} v$ gives another tableau $\hat{t}$ of shape $\mu$, with $\lambda\subseteq \mu$. Let $S$ be the associated recording tableau, of shape $\mu \setminus \lambda$. Then $\text{Rect}(S) = Q$.
%\end{prop}

We can now define a sign statistic on the $Q$-tableaux which will allow us to discriminate in which plethysm the associated copy of $s_\nu $ lies.

Let $Q$ be a tableau of shape $\nu\vdash 2|\lambda|$ and content $\lambda^2$, for a fixed partition $\lambda = (\lambda_1, \ldots, \lambda_\ell)$. Let $Q^{(i)}$ be the subtableau of $Q$ with entries $2i-1$ and $2i$. Then $Q_i = \text{Rect}(Q^{(i)})$ is of shape $(2\lambda_i-j_i, j_i)$, for $1\leq i \leq \ell$.\\ Define the sign of $Q$ to be $\text{sign}(Q) = \prod_{i=1}^{\ell} (-1)^{j_i}$.\\

\begin{theo*}[Part 1]
The copy of $s_\nu$ indexed by $Q$ in $h_\lambda^2$ lies in $s_2[h_\lambda]$ if $\text{sign}(Q) = 1$, and in $s_{11}[h_\lambda]$ if $\text{sign}(Q) = -1$.
\end{theo*}

This is proved in the next section, where we show that this sign statistic is well defined, %records all the right information about a specific copy of $s_\nu$ 
and can be used to discriminate contribution to a plethysm. %The first part of theorem 1 give us also the first part of our second theorem : 
%Let's first state the following :
We also have :

\begin{theo*}[Part 1]
Consider the following numbers:
\begin{itemize}
    \item $(K^{\nu}_{\lambda^2})^{+} = \#\{ Q \text{ tableau of shape } \nu \text{ and content } \lambda^2 \ | \ \text{sign}(Q)=1 \}$ 
    \item $(K^{\nu}_{\lambda^2})^{-} = \#\{ Q \text{ tableau of shape } \nu \text{ and content } \lambda^2 \ | \ \text{sign}(Q)=-1 \}$ 
\end{itemize}

Then we have: 

\vspace{-2em}
\begin{align*}
s_2[h_{\lambda}] &= \displaystyle \sum_{\nu} (K^{\nu}_{\lambda^2})^{+} s_{\nu} \\
s_{11}[h_{\lambda}] &= \displaystyle \sum_{\nu} (K^{\nu}_{\lambda^2})^{-} s_{\nu} 
\end{align*}
\end{theo*}

\begin{ex}
Let's see how theorem 1 in action. The four tableaux of shape $\nu = (3,2,1)$ and content $(2,1)^2 = (2,2,1,1)$ are illustrated below. We thank Franco Saliola for providing the figure.  The computed signs tell us that $s_2[h_{21}]$ and $s_{11}[h_{21}]$ both contain two copies of $s_{321}$.

\vspace{2em}
\hspace{-3.5em}
\begin{tabular}{cccccrc}
    $Q$ & $Q^{(1)}$ & $\mathrm{Rect}(Q^{(1)})$ & $Q^{(2)}$ & $\mathrm{Rect}(Q^{(2)})$ & $\mathrm{sign}(Q)$ & contributes to\\ \hline
    \\
    $\young(112,23,4)$ & $\young(112,2)$ & $\young(112,2)$ & $\gyoung(:;3,;4)$ & $\young(3,4)$ & $(-1)^{1+1} = 1$ & $s_{2}[h_{21}]$
    \\[2ex]
    $\young(112,24,3)$ & $\young(112,2)$ & $\young(112,2)$ & $\gyoung(:;4,;3)$ & $\young(34)$ & $(-1)^{1+0} = -1$ & $s_{11}[h_{21}]$
    \\[2ex]
    $\young(113,22,4)$ & $\young(11,22)$ & $\young(11,22)$ & $\gyoung(::;3,;4)$ & $\young(3,4)$ & $(-1)^{2+1} = -1$ & $s_{11}[h_{21}]$
    \\[2ex]
    $\young(114,22,3)$ & $\young(11,22)$ & $\young(11,22)$ & $\gyoung(::;4,;3)$ & $\young(34)$ & $(-1)^{2+0} = 1$ & $s_{2}[h_{21}]$
\end{tabular}

\end{ex}

\subsection{Proof of theorem 1}
\label{section:proofs}

In order to prove Theorem~\ref{thm:signTableauQ}, we need to introduce basic properties of plethysm.%, and some transition matrices between different basis of symmetric functions.

The plethysm of two symmetric functions %correspond to a monomial substitution of symmetric functions, which 
can be defined using the following properties and power sum symmetric functions \cite{Macdonald}.\\ %since all symmetric function can be expressed in term of this basis.
%Plethysm then is defined by the following properties, where 
Let $f,g,h$ be symmetric functions :
\begin{itemize}
    \item $p_k[p_m] = p_{km}$
    \item $p_k[f + g] = p_k[f] + p_k[g]$
    \item $p_k[f\cdot g] = p_k[f] \cdot p_k[g]$
    \item $(f + g)[h] = f[h] + g[h]$
    \item $(f\cdot g)[h] = f[h] \cdot g[h]$.
\end{itemize}

The proof that the sign $(-1)^{j_i}$ of each $Q_i$ is well defined rests on the following result introduced by Littlewood in \cite{Littlewood2}.

\begin{prop}

%We have
%\[(h_n)^2 = \sum_{j=0}^n s_{(2n-j,j)},\]
%\[(e_n)^2 = \sum_{j=0}^n s_{(2^{n-j},1^{2j}}).\]

The plethysms $s_2[h_n]$, $s_{11}[h_n]$ in $h_n^2$ decompose into the Schur basis as 
\[
s_{2}[h_n] = h_{2}[h_n] = \sum_{j=0}^{\lfloor \frac{n}{2} \rfloor} s_{(2n-2j,2j)},
\]
\[
s_{11}[h_n] = e_{2}[h_n] = \sum_{j=0}^{\lfloor \frac{n-1}{2} \rfloor} s_{(2n-(2j+1),(2j+1))}.
\]
%
%The plethysms $s_2[e_n]$, $s_{11}[e_n]$ in $e_n^2$ decompose into the Schur basis as 
%\[s_{2}[e_n] = h_{2}[e_n] = \sum_{j=0}^{\lfloor \frac{n}{2} \rfloor} s_{(2^{n-2j},1^{4j})},\]
%
%\[s_{11}[e_n] = e_{2}[e_n] = \sum_{j=0}^{\lfloor \frac{n-1}{2} \rfloor} s_{(2^{n-(2j+1)},1^{2(2j+1)})}.\]
\label{thm:ShapeEvenOdd}
\end{prop}

\begin{proof}[Idea of proof]
%\flo{(Version article)}
The proof %shown below is rather simple, and 
follows directly from the transition matrix between the different basis, and basis properties of plethysm. It can be found in chapter 1, section 8, of Macdonald \cite{Macdonald}.  %but seems to be hard to find in the literature. 
%
%\flo{(Version ArXiV)}
For completeness, we include it here.
%The two first expressions follows from proposition \ref{Power}. For the others, they 
The following formula is proved in MacDonald \cite{Macdonald} :
\[
p_2[h_n] = \displaystyle \sum_{j=0}^n (-1)^j s_{(2n-j,j)}
\]
%\[p_2[e_n] = \displaystyle \sum_{j=0}^n (-1)^j s_{2^{n-j},1^{2j}}\]
%\[
%h_2 = \frac{1}{2}p_1^2 + \frac{1}{2}p_2
%\]
%\[e_2 = \frac{1}{2}p_1^2 - \frac{1}{2}p_2\]

For $s_2[h_n]$, as $s_2 = h_2 = \frac{1}{2}p_1^2 + \frac{1}{2}p_2$, we have:
\begin{align*}
h_2[h_n]    &= \left( \frac{1}{2}p_1^2 + \frac{1}{2}p_2 \right)[h_n] \\
           &= \frac{1}{2}p_1^2[h_n] + \frac{1}{2}p_2[h_n] \\
           &= \frac{1}{2}h_n^2 + \frac{1}{2}p_2[h_n] \\
           &= \frac{1}{2}\displaystyle \sum_{i=0}^n s_{(2n-i,i)} + \frac{1}{2}\displaystyle \sum_{i=0}^n (-1)^i s_{(2n-i,i)} \\
        &= \sum_{j=0}^{\lfloor \frac{n}{2} \rfloor} s_{(2n-2j,2j)}
\end{align*}
The other expression is obtained in the same way.
\end{proof}

Applying this result in our context, we obtain: 

\begin{cor}
Let $t = (t_1, t_2)$ be a $(n)^2$-tuple of row tableaux, $W_t$ the associated biword and $(P,Q)$ the associated pair of tableaux of the same shape $(2n-j,j)$, with $0 \leq j \leq n$. The associated monomial $x^P = x^t$ appears in $s_{(2n-j,j)}$ in $h_n^2$, %. The monomial $x^T = x^s x^t$ 
which contributes to $s_2[h_n]$ if $\text{sign}(Q) = (-1)^j$ is positive, and to $s_{11}[h_n]$ if $\text{sign}(Q)$ is negative.
\end{cor}

The other ingredient needed to prove our theorem is the following result:

\begin{prop}
Let $g_1,g_2, \ldots, g_n$ be any symmetric functions. Then:
\[
(g_1g_2\ldots g_n)^2 = s_2[g_1g_2\ldots g_n] + s_{11}[g_1g_2\ldots g_n],
\]
where
\[
s_2[g_1g_2\ldots g_n] = \sum_{\substack{I\subseteq \{ 1, \ldots , n \} \\ |I| \text{ even }}} \prod_{i\in I} s_{11}[g_i]\cdot \prod_{j\in I^c} s_{2}[g_j],
\]
\[s_{11}[g_1g_2\ldots g_n] = \sum_{\substack{I\subseteq \{ 1, \ldots , n \} \\ |I| \text{ odd }}} \prod_{i\in I} s_{11}[g_i]\cdot \prod_{j\in I^c} s_{2}[g_j].
\]
\label{thm:SymAntiSym}
\end{prop}

\begin{proof}

As $g^2 = s_2[g]+s_{11}[g]$,
 we have
\begin{align*}
(g_1g_2\ldots g_n)^2    &= g_1^2g_2^2\ldots g_n^2 \\
                        &= \prod_{i=1}^n (s_2[g_i] + s_{11}[g_i])
\end{align*}

We will prove the result by induction on $n$. 

For general symmetric functions $g_1, g_2$, it is possible to prove that  $s_2[g_1g_2] = s_{2}[g_1]s_2[g_2] + s_{11}[g_1] s_{11}[g_2]$, while $s_{11}[g_1g_2] = s_{11}[g_1]s_2[g_2] + s_{2}[g_1] s_{11}[g_2]$, using only basic properties of plethysm (see for example \cite{Macdonald}). %This gives the desired result. %It is possible to prove these results using only basic properties of plethysm.

Suppose now we have the result for $n$ symmetric functions, and we show it holds for $n+1$ symmetric functions $g_1, g_2, \ldots, g_{n+1}$. 

Recall that the product of symmetric functions $g_1, g_2, \ldots, g_n$ is symmetric.
\begin{align*}
s_2[g_1 g_2\ldots g_n g_{n+1}] %=& \sum_{\nu\vdash 2} s_\nu[g_1 g_2\ldots g_n]s_\nu [g_{n+1}] \text{ since $g_1 g_2 \ldots g_n$ is a symmetric function} \\
=& s_{2}[g_1 g_2\ldots g_n]s_2[g_{n+1}] + s_{11}[g_1 g_2\ldots g_n]s_{11}[g_{n+1}] \\
=& \left( \sum_{\substack{I\subseteq \{ 1, \ldots , n \} \\ |I| \text{ even }}} \prod_{i\in I} s_{11}[g_i]\cdot \prod_{j\in I^c} s_{2}[g_j] \right) s_2[g_{n+1}] \\
&+ \left( \sum_{\substack{I\subseteq \{ 1, \ldots , n \} \\ |I| \text{ odd }}} \prod_{i\in I} s_{11}[g_i]\cdot \prod_{j\in I^c} s_{2}[g_j] \right) s_{11}[g_{n+1}]\\
=& \sum_{\substack{I\subseteq \{ 1, \ldots , n+1 \} \\ |I| \text{ even }}} \prod_{i\in I} s_{11}[g_i]\cdot \prod_{j\in I^c} s_{2}[g_j].
\end{align*}

The proof for $s_{11}[g_1 g_2\ldots g_n g_{n+1}]$ is similar.
\end{proof}

\begin{proof}[Proof of Theorem 1]
As seen in proposition \ref{thm:ShapeEvenOdd}, the shape $(2\lambda_i-j_i,j_i)$ of the $Q_i$ and the sign $(-1)^{j_i}$ determine to which plethysm the associated Schur function $s_{(2\lambda_i-j_i,j_i)}$ contributes in $h_{\lambda_i}^2$.
%
%For a given tableau $Q$ of shape $\nu$ and content $\lambda^2$, each $Q_i$ does this. Therefore 
$Q$ then encodes that the associated Schur function $s_\nu$ appears in the product $s_{(2\lambda_1-j_1,j_1)} s_{(2\lambda_2-j_2,j_2)} \ldots s_{(2\lambda_\ell-j_\ell,j_\ell)}$ in $h_\lambda^2 = h_{\lambda_1}^2h_{\lambda_2}^2 \ldots h_{\lambda_\ell}^2 = \prod_{i=1}^\ell (s_2[h_{\lambda_i}] + s_{11}[h_{\lambda_i}])$. 

Let the set $I = \{ \ i \ | \ j_i \text{ odd }   \}$ record the anti-symmetric parts that have been picked.% in the product $h_\lambda^2 = \prod_{i=1}^\ell (s_2[h_{\lambda_i}] + s_{11}[h_{\lambda_i}])$. 
The sign of $Q$ relies only on the parity of the cardinality of $I$ :  $\text{sign}(Q) = \prod_i \text{sign}(Q_i) = \prod_i (-1)^{j_i} = (-1)^{|I|}$.

Therefore, the sign of such a tableau $Q$ is well defined and indicates the participation of the associated Schur function to a plethysm.%, and $Q$ records the product of smaller Schur function in which the associated Schur function appears.
\end{proof}

We can use the same ideas to show a more general result. Let $\mu$ be another partition, and consider the product $s_{\mu}h_{\lambda}^2$. Iteratively applying the Pieri rule, as in proposition \ref{Power}, we have
\[
s_{\mu}h_{\lambda}^2 = \displaystyle \sum_{\nu \vdash 2|\lambda| + |\mu|} K^{\nu / \mu}_{\lambda^2} s_{\nu},
\]
where $K^{\nu / \mu}_{\lambda^2}$ counts the number of tableaux of shape $\nu / \mu$ and content $\lambda^2$. Using the very same ideas and notations developed in this section, we obtain the following corollary:

\begin{cor}
Let $Q$ be a tableau of shape $\nu/\mu$ and content $\lambda^2$, for $\nu\vdash 2|\lambda| + |\mu|$, and fixed partitions $\mu$ and $\lambda = (\lambda_1, \ldots, \lambda_\ell)$. Let $Q^{(i)}$ be the subtableau of $Q$ with entries $2i-1$ and $2i$. Then, $Q_i = \text{Rect}(Q^{(i)})$ is of shape $(2\lambda_i-j_i, j_i)$, for $1\leq i \leq \ell$. Define the sign of $Q$ to be $\text{sign}(Q) = \prod_{i=1}^{\ell} (-1)^{j_i}$. \\

The copy of $s_\nu$ indexed by $Q$ in $s_{\mu}h_\lambda^2$ lies in $s_{\mu}(s_2[h_\lambda])$ if $\text{sign}(Q) = 1$, and in $s_{\mu}(s_{11}[h_\lambda])$ if $\text{sign}(Q) = -1$.
\end{cor}

\section{Associating a $\widetilde{Q}$-tableau indexing a copy of $s_\nu$ in $e_\lambda^2$ to a plethysm}
\label{section:e}
\setcounter{theo*}{0}
All of the above can be translated into the realm of elementary symmetric functions.
We have seen that the definition of elementary symmetric functions is very similar to that of homogeneous symmetric functions. Their Pieri rules are conjugate, and so applying them iteratively give conjugate tableaux.  

Therefore $e_\lambda^2 = \displaystyle \sum_\nu K^{\nu'}_{\lambda^2} s_{\nu}$, where $K^{\nu'}_{\lambda^2}$ is the number of tableaux of conjugate shape $\nu'$ and content $\lambda^2$.

We use a variation of the RSK algorithm to define the bijection of $\lambda$-tuples of column tableaux and pairs of "tableaux" of the same shape. This variant is called the RSK' algorithm in \cite{Krattenthaler}, but to avoid confusion with the conjugate, we call it the $\widetilde{\text{RSK}}$ algorithm. 

It works in the same way as RSK, but is defined on \textit{Burge words} : biwords $\widetilde{W}$ such that no bi-letters appear twice, and with its bi-letters $\big($\begin{tabular}{c} $u_i$ \\ $v_i$ \\ \end{tabular}$\big)$ ordered such that the $u_i$ weakly increase (as before), but for $u_i = u_{i+1}$, $v_i > v_{i+1}$.\\ 

The pair of "tableaux" of the same shape $(\widetilde{P},\widetilde{Q}) = \widetilde{\text{RSK}}(\widetilde{W}_t)$ obtained is such that the conjugate $\widetilde{Q}'$ of $\widetilde{Q}$ is a semistandard tableau. We call $\widetilde{Q}$ a conjugate tableau. The tableau $\widetilde{Q}'$ then has shape conjugate to that of $\widetilde{P}$. The $\widetilde{RSK}$ algorithm establishes a bijection between Burge words and pairs tableaux/conjugate tableaux of the same shape. Note that the bijection between $\lambda$-tuples of column tableaux and Burge words follows the same construction we have seen before, with the top word being $\lambda$, and the bottom word being the reading words of the column tableaux.

\begin{ex}
Lef $t = (\young(1,2,3,5,7),\young(1,3,4,6,8),\young(2,3,5),\young(1,2,4))$ be a $(5,3)^2$-tuple of column tableaux. 
\vspace{-1em}
Then $\widetilde{W}_t = \big( \begin{tabular}{cccc} 11111&22222& 333 & 444 \\ 75321 & 86431 & 532 & 421 \\ \end{tabular}\big)$ and $(\widetilde{P},\widetilde{Q}) = (\young(1112,2234,335,46,58,7),\young(1234,1234,123,12,12,4) )$.
\end{ex}

We have that 
\begin{cor}
Let $t$ a $\lambda^2$-tuple of column tableaux, $\widetilde{W}_t$ the associated Burge word and $(\widetilde{P},\widetilde{Q}) = \widetilde{\text{RSK}}(\widetilde{W}_t)$. Let $(\widetilde{T_i},\widetilde{Q_i}) = \widetilde{\text{RSK}}\big(  \begin{tabular}{cccccccc} $ 2i-1 \ldots 2i-1 $ & $ 2i \ldots 2i $ \\ word($t_{2i-1}$) & word($t_{2i}$) \end{tabular}  \big)$, and $\widetilde{Q^{(i)}}$ the subtableau of $\widetilde{Q}$ containing entries $2i-1$ and $2i$, for $1\leq i \leq \ell(\lambda)$.

Then $\widetilde{Q_i}' = \text{Rect}((\widetilde{Q^{(i)}})')$.
\end{cor}

\begin{proof}
In order to get this result, we need to adapt the proof of Fulton's result (\cite{Fulton}, chapter 5) to $\widetilde{\text{RSK}}$. We can invert all bi-letters of a Burge word and rearrange them in lexicographical order (so that it is a biword). This gives a bottom word such that its reverse is Knuth-equivalent to the reading words of both $\widetilde{Q_i}'$ and $\widetilde{Q^{(i)}}'$. It is the same strategy that Fulton uses in his proof of the RSK case.
The final rectification is applied on the conjugate of the conjugate tableau $\widetilde{Q^{(i)}}$, because jeu de taquin is not well defined otherwise. 
\end{proof}

This allows us to also define a sign statistic on the $\widetilde{Q}$-tableaux which will allow us to discriminate in which plethysm the associated copy of $s_\nu $ lies :

\begin{theo*}[Part 2]
Let $\widetilde{Q}$ be a conjugate tableau of shape $\nu\vdash 2|\lambda|$ and content $\lambda^2$, for a fixed partition $\lambda = (\lambda_1,\ldots, \lambda_\ell)$. Let $\widetilde{Q^{(i)}}$ be the subtableau of $\widetilde{Q}$ with entries $2i-1$ and $2i$, for $1\leq i \leq \ell$. Then $\widetilde{Q_i} = (\text{Rect}(\widetilde{Q^{(i)}}'))'$ is of shape $(2^{\lambda_i-j_i}, 1^{2j_i})$. Define the (conjugate) sign of $\widetilde{Q}$ to be $\text{sign}_c(\widetilde{Q}) = \prod_{i=1}^{\ell} (-1)^{j_i}$.\\

The copy of $s_\nu$ indexed by $\widetilde{Q}$ in $e_\lambda^2$ lies in $s_2[e_\lambda]$ if $\text{sign}_c(\widetilde{Q}) = 1$, and in $s_{11}[e_\lambda]$ if $\text{sign}_c(\widetilde{Q}) = -1$.
\end{theo*}

This give us the following theorem, second part of theorem 2.

\begin{theo*}[Part 2]
Consider the following numbers:
\begin{itemize}
     \item $(K^{\nu'}_{\lambda^2})^{+} = \#\{ \widetilde{Q}' \text{ tableau of shape } \nu' \text{ and content } \lambda^2 \ | \ \text{sign}_c(\widetilde{Q})=1 \}$ 
     \item $(K^{\nu'}_{\lambda^2})^{-} = \#\{ \widetilde{Q}' \text{ tableau of shape } \nu' \text{ and content } \lambda^2 \ | \ \text{sign}_c(\widetilde{Q})=-1 \}$ 
\end{itemize}

Then we have: 

\vspace{-2em}
\begin{align*}
s_2[e_{\lambda}] &= \displaystyle \sum_{\nu} (K^{\nu'}_{\lambda^2})^{+} s_{\nu} \\
s_{11}[e_{\lambda}] &= \displaystyle \sum_{\nu} (K^{\nu'}_{\lambda^2})^{-} s_{\nu}
\end{align*}
\end{theo*}

The proof that the sign of each $Q_i$ is well-defined rests on the following result analogous to that seen in section~\ref{section:proofs}, also introduced initially by Littlewood.

\begin{prop}
The plethysms $s_2[e_n]$, $s_{11}[e_n]$ in $e_n^2$ decompose into the Schur basis as 
\[
s_{2}[e_n] = h_{2}[e_n] = \sum_{j=0}^{\lfloor \frac{n}{2} \rfloor} s_{(2^{n-2j},1^{4j})},
\]

\[
s_{11}[e_n] = e_{2}[e_n] = \sum_{j=0}^{\lfloor \frac{n-1}{2} \rfloor} s_{(2^{n-(2j+1)},1^{2(2j+1)})}.
\]
\label{thm:ShapeEvenOdde}
\end{prop}

\begin{proof}
The proof is very similar as the one in section~\ref{section:proofs} and relies on the following formulas also proven in MacDonald \cite{Macdonald} :
\[
p_2[e_n] = \displaystyle \sum_{j=0}^n (-1)^j s_{2^{n-j},1^{2j}}
\]
\[
e_2 = \frac{1}{2}p_1^2 - \frac{1}{2}p_2
\]
\end{proof}

\begin{proof}[Proof of part 2 of Theorem 2]
As seen in \ref{thm:ShapeEvenOdde}, the shape $(2^{\lambda_i-j_i},1^{2j_i})$ of the $\widetilde{Q_i}$ and the (conjugate) sign $(-1)^{j_i}$ determine in which plethysm the associated Schur function $s_{(2^{\lambda_i-j_i},1^{2j_i})}$ contributes to in $e_{\lambda_i}^2$.

For a given conjugate tableau $\widetilde{Q}$ of shape $\nu$ and content $\lambda^2$, this holds for each $\widetilde{Q_i}$, and 
$\widetilde{Q}$ then encodes that the associated Schur function $s_\nu$ appears in the product $s_{(2^{\lambda_1-j_1},1^{2j_1})} s_{(2^{\lambda_2-j_2},1^{2j_2})} \ldots s_{(2^{\lambda_\ell-j_\ell},1^{2j_\ell})}$ in $e_\lambda^2 = e_{\lambda_1}^2e_{\lambda_2}^2 \ldots e_{\lambda_\ell}^2 = \prod_{i=1}^\ell (s_2[e_{\lambda_i}] + s_{11}[e_{\lambda_i}])$. \\

Let the set $I = \{ \ i \ | \  j_i \text{ odd }   \}$ record the anti-symmetric parts that have been picked. %in the product $\prod_{i=1}^\ell (s_2[e_{\lambda_i}] + s_{11}[e_{\lambda_i}])$. 
The (conjugate) sign of $\widetilde{Q}$ relies only on the parity of the cardinality of $I$ as before :  $\text{sign}_c(\widetilde{Q}) = \prod_i \text{sign}_c(\widetilde{Q_i}) = \prod_i (-1)^{j_i} = (-1)^{|I|}$.
\end{proof}

As in section 3, we can generalize these ideas to the product $s_{\mu}e_{\lambda^2}$:

\begin{cor}
Let $\widetilde{Q}$ be a conjugate tableau of shape $\nu/\mu$ and content $\lambda^2$, for $\nu\vdash 2|\lambda| + |\mu|$, and fixed partitions $\mu$ and $\lambda = (\lambda_1, \ldots, \lambda_\ell)$. Let $\widetilde{Q^{(i)}}$ the subtableau of $\widetilde{Q}$ with entries $2i-1$ and $2i$, for $1\leq i \leq \ell$ and $\widetilde{Q_i} = (\text{Rect}(\widetilde{Q^{(i)}}'))'$ of shape $(2^{\lambda_i-j_i}, 1^{2j_i})$. Define the (conjugate) sign of $\widetilde{Q}$ to be $\text{sign}_c(\widetilde{Q}) = \prod_{i=1}^{\ell} (-1)^{j_i}$.\\

The copy of $s_\nu$ indexed by $\widetilde{Q}$ in $s_{\mu}e_\lambda^2$ lies in $s_{\mu}(s_2[e_\lambda])$ if $\text{sign}_c(\widetilde{Q}) = 1$, and in $s_{\mu}(s_{11}[e_\lambda])$ if $\text{sign}_c(\widetilde{Q}) = -1$.
\end{cor}

\section{What's next?}

We are of course interested in investigating higher plethysms. 

\subsection{Generalizing this approach}

To be able to generalize our approach, we need to use some known results, which are implicit in \cite{Macdonald}.

\begin{prop}\label{PlethsymDecomposition}
For any symmetric function $g$,
\[
g^n = \displaystyle \sum_{\mu \vdash n} f^{\mu} s_{\mu}[g],
\]
where $f^{\mu}$ is the number of standard tableaux of shape $\mu$.
\end{prop}

%The proof is straightforward and uses only basic properties of plethysm and the RSK algorithm. For completeness we include it here.

\begin{proof}
We have that $p_1[g] = g$ for any symmetric function $g$. Moreover, $p_1^n = f^{\mu}s_{\mu}$, which follows directly from the RSK algorithm: the indices picked in each copy of $p_1$ in $p_1^n = (x_1 + x_2 + x_3 +\ldots)^n$ form a word of length $n$, which are in bijection with pairs of tableaux of the same shape $\mu\vdash n$, with the recording tableau being standard. % (with entries $1$ to $n$, each appearing once).  

Using plethysm rules, we have:
\begin{align*}
g^n     & = (p_1[g])^n \\
        & = p_1^n[g] \\
        & = \displaystyle \sum_{\mu \vdash n} f^{\mu} s_{\mu}[g]
\end{align*}

\end{proof}

We can also use the Pieri rule recursively (or RSK/$\widetilde{RSK}$) to get the decomposition of $h_\lambda^n$ and $e_\lambda^n$ into the basis of Schur functions:

\begin{prop}
$(h_{\lambda})^n = h_{\lambda^n} = \displaystyle \sum_{\nu \vdash n|\lambda|} K_{\lambda^n}^{\nu} s_{\nu}$, and $(e_{\lambda})^n = e_{\lambda^n} = \displaystyle \sum_{\nu \vdash n|\lambda|} K_{\lambda^n}^{\nu'} s_{\nu}$,
where the Kostka numbers $K_{\lambda^n}^\nu$ count the number of tableaux of shape $\nu$ and content $\lambda^n$.
%We also have:
%$(e_{\lambda})^n = e_{\lambda^n} = \displaystyle \sum_{\nu \vdash n|\lambda|} \widetilde{K_{\lambda^n}^{\nu}} s_{\nu}$, where $\widetilde{K_{\lambda^n}^{\nu}}$ counts conjugate tableaux of shape $\nu$ and content $\lambda^n$.
\end{prop}

All of the above constructions %discussion of expressing the tableaux of shape $\nu$ and content $\lambda^2$ counted by $K^\nu_{\lambda^2}$ using RSK 
can be generalized to $n>2$.\\

\newpage
For a tableau $Q$ of shape $\nu$ and content $\lambda^n$, we can consider the subtableaux $Q^{(i)}$ of $Q$ with entries $ni-(n-1), ni-(n-2), \ldots, ni$, for $1\leq i \leq \ell(\lambda)$, and their rectifications $Q_i$. % which stand for recording tableaux of sub-biwords. 
If $\nu_i = shape(Q_i)$, and $\nu = shape(Q)$, then the Schur function $s_\nu$ occurs in the product $\prod_i s_{\nu_i}$, each $s_{\nu_i}$ indexed by $Q_i$ in $h_{\lambda_i}^2$.
Same thing goes to describe the tableaux counted by $K^{\nu'}_{\lambda^n}$ using $\widetilde{RSK}$. \\

We would like to identify a statistic on the tableaux $Q_i$ and $Q$ to determine to which plethysms of $h_{\lambda_i}^n$ and $h_\lambda^n$ the associated Schur functions $s_{\nu_i}$ and $s_\nu$ contribute, and similarly for the plethysms of $e_{\lambda_i}^n$ and $e_\lambda^n$. However, it proves to be extremely difficult.\\ 

For $n = 3$, we have that $g^3 = s_3[g]+ 2s_{21}[g]+s_{111}[g]$. The changing multiplicities already make things a bit complicated. For $h_n^3$, and $e_n^3$, we have closed formulas for the number of copies of $s_\nu$ that lie in $s_3[h_n], s_{111}[h_n]$ (resp. $s_3[e_n]$ and $s_{111}[e_n]$), studied by Chen \cite{Chen} and Thrall \cite{Thrall}. The leftover Schur functions would then be dispatched into the two copies of $s_{21}[h_n]$ (resp. $s_{21}[e_n]$). This can help pinpoint the right statistic to consider. We would like to refine this by understanding more precisely which copies of $s_\nu$ lie in each plethysm, eventually also distinguishing between the two copies of $s_{21}[h_n]$ or $s_{21}[e_n]$, but we are far from this result.\\

In the next section, we describe another strategy which might prove successful and which allows to recover our results for $n = 2$.

\subsection{Generalizing using ribbon tableaux}

The title of this article refers explicitly to that of Carré and Leclerc \cite{CarreLeclerc}, where they describe the product of two Schur functions in terms of domino tableaux.
%They proved that pairs of tableaux $(t_1,t_2)$ of shape $(\mu,\nu)$ are in bijection with domino tableaux of shape $I$ such that $(\mu,\nu)$ is the $2$-quotient of $I$. 
They showed that the number of Yamanouchi domino tableaux of a certain shape $I$ and content $\lambda$ give the multiplicity $c_{\mu \nu}^\lambda$ of $s_\lambda$ in $s_\mu s_\nu$. When $\mu = \nu$, then $I = (2\mu)^2 = (2\mu_1,2\mu_1, 2\mu_2, 2\mu_2, \ldots, 2\mu_\ell, 2\mu_\ell)$. A Yamanouchi domino tableau has a Yamanouchi (or reverse lattice) reading word: the content of each suffix is a partition. Domino tableaux are read off by scanning rows left to right and bottom to top, reading dominoes on their first scanning.\\ %, reading rows left to right and bottom to top. \\

\newpage
For a square $s_\mu
^2$, Carré and Leclerc showed that the parity of the cospin of the Yamanouchi domino tableaux determine whether the associated Schur function $s_\lambda$ contributes to $s_2[s_\mu]$ %its symmetric 
or $s_{11}[s_\mu]$, with the cospin equal the number of horizontal dominoes divided by two. %anti-symmetric part. 
%The cospin of a domino tableau (appearing in a square) is the number of horizontal dominoes it contains. More generally, 
Note that the cospin is more generally defined as the maximal number of vertical dominoes that can pave $I$ minus the number of vertical dominoes in the domino tableau, divided by two. %Conveniently, in the case of a square $s_\mu^2$, the cospin equals the number of horizontal dominoes divided by two, since $I$ can be entirely paved by vertical dominoes. The interest of the cospin over the spin (half the number of vertical dominoes) is that the cospin is always an integer, as the parity of the number of vertical dominoes that can appear depends only on the shape $I$. %The spin (and cospin) statistic can be more generally defined for $r$-ribbon tableaux (domino tableaux are $2$-ribbon tableaux) as half the sum of the height of all ribbons, where the height of a ribbon in the number of rows it spans minus one.

For $h_n^2 = s_n^2$, $I=(2n,2n)$. The only pavings and fillings of $I$ that result in Yamanouchi domino tableaux consist of a sequence of $2n-2j$ vertical dominos filled with ones, followed by $j$ pairs of stacked horizontal dominoes, filled with a one for the top one, and a two for the bottom one. The weight of such a Yamanouchi tableau is then $\mu = (2n-j,j)$, and its cospin is $j$. By the result of Carré and Leclerc, there is a unique copy of $s_\mu$ in the decomposition of $h_n^2 = s_n^2$, and it contributes to the symmetric part of the square if $j$ is even, and to its anti-symmetric part if $j$ is odd. 

For $e_n^2 = s_{1^n}^2$, $I=(2^{2n})$. The only pavings and fillings of $I$ that result in Yamanouchi domino tableaux, built top to bottom, consists of a sequence of $n-j$ couples of vertical dominoes filled with ones, then two's, etc. till $n-j$, followed by $2j$ horizontal dominoes, filled with $j+1$, etc. till $n+j$. The weight of such a Yamanouchi tableau is then $\mu = (2^{n-j},1^{2j})$, and its cospin is $j$. Therefore, by the result of Carré and Leclerc, there is a unique copy of $s_\mu$ in the decomposition of $e_n^2 = s_{1^n}^2$, and it contributes to the symmetric part of the square if $j$ is even, and to its anti-symmetric part if $j$ is odd.\\

The Yamanouchi domino tableaux are illustrated below.\\

\hspace{-6em}
\Yboxdim{0.39cm}
\scriptsize{
\begin{tabular}{ccccc}
\gyoung(|21|21|21|21|2<:>|21|21|21|21:,:)& 
\gyoung(|21|21|21|21|2<:>|21|21_21:,:::::::_22)& $\ldots$& 
\gyoung(|21|21_21|2<:>_21_21:,::_22:_22_22)& 
\gyoung(_21_21|2<:>_21_21:,_22_22:_22_22)\\
& & & & \\
%$word  : 1^{2n}$ & $word : 1^{2n-1} 2 1$ & $\ldots$ & $word : 1^2 2^{n-1} 1^{n-1}$ & $word : 2^{n} 1^{2n}$ \\
$\mu = (2n)$ & $\mu = (2n-1,1)$ & $\ldots$ & $\mu = (n+1,n-1)$ & $\mu = (n,n)$. \\
& & & & \\
\gyoung(|21|21,|22|22,_2<:>,|2<{\tiny n-1}>|2<{\tiny n-1}>,|2<n>|2<n>,:,:,:,:) & \gyoung(|21|21,|22|22,_2<\ldots>,|2<{\tiny n-1}>|2<{\tiny n-1}>,_2<n>,_2<n+1>,:,:,:)
& $\ldots$ &  \gyoung(|21|21,_22,_23,_2<\ldots>,_2<2n-4>,_2<2n-3>,_2<2n-2>,_2<2n-1>,:) &
\gyoung(_21,_22,_23,_24,_2<\ldots>,_2<2n-3>,_2<2n-2>,_2<2n-1>,_2<2n>)\\
&&&&\\
%$word  : n^2 (n-1)^2 \ldots 2^2 1^2$ & $word : (n+1) n (n-1)^2 \ldots 2^2 1^2$ & $\ldots$ & $word : (2n-1) (2n-2) \ldots 3 2 1^2$ & $word : (2n) (2n-1) \ldots 3 2 1$ \\
$\mu = (2^{n})$ & $\mu =  (2^{n-1},1^2)$ & $\ldots$ & $\mu = (2^1,1^{2(n-1)})$ & $\mu = (1^{2n})$\\
&&&&\\
\small $cospin = 0$ & \small $cospin = 1$ & \small $\ldots$ & \small $cospin = n-1$ & \small $cospin = n$ \\
\end{tabular}
}
\normalsize

\vspace{1em}

This gives an alternate proof of proposition~\ref{thm:ShapeEvenOdd}.\\

This could be generalized to higher plethysms. We know, thanks to Stanton and White \cite{StantonWhite}, that $r$-tuples of tableaux are in bijection with $r$-ribbon tableaux (domino tableaux are $2$-ribbon tableaux). %That correspondence most probably does not agree with crystal theory as we've seen with the Carré-Leclerc correspondence, but still points towards a crystal structure on $r$-ribbon tableaux. 
In particular, $(m)^n$-tuples of row (resp. column) tableaux are in bijection with $r$-ribbon tableaux of shape $(m\cdot n)^n$ (resp. $(n)^{m\cdot n}$). Could we then define a Yamanouchi-like property and a cospin-like statistic that could help solve the plethystic decomposition of $h_m^n$ and $e_m^n$? Could this be extended even to $h_\lambda^n$, $e_\lambda^n$ or $s_\lambda^n$?

\section*{Acknowledgements}
The authors are grateful to Franco Saliola for his support. FMG received support from NSERC.
%We thank the anonymous reviewer for comments to help improve exposition.

\bibliographystyle{alpha}
\bibliography{main.bib}

\end{document}